\documentclass{amsart}
\usepackage[cp1251]{inputenc}
\usepackage[english]{babel}
\usepackage{amsmath}
\usepackage{amssymb}
\usepackage{amsfonts}
\newtheorem{lemma}{Lemma}
\newtheorem{theorem}{Theorem}

\newtheorem{remark}{Remark}

\begin{document}

\begin{center} {\bf ON TRANSPARENT BOUNDARY CONDITIONS FOR THE
HIGH-ORDER HEAT EQUATION}
\end{center}

\vspace{5mm}

\begin{center}
{\bf Durvudkhan Suragan, Niyaz Tokmagambetov}
\end{center}

\vspace{1mm}

\begin{center}
Almaty, Kazakhstan
\end{center}

\vspace{4mm}

{\bf Abstract.} In this paper we develop an artificial initial
boundary value problem for the high-order heat equation in a
bounded domain $\Omega$. It is found an unique classical solution
of this problem in an explicit form and shown that the solution of
the artificial initial boundary value problem is equal to the
solution of the infinite problem (Cauchy problem) in $\Omega$.

\vspace{4mm}

 {\bf keywords.} transparent boundary conditions \*\
an artificial initial boundary value problem \*\ a high-order
parabolic equation

\vspace{4mm}

{\bf MSC}:{AMS Mathematics Subject Classification (2000) numbers:
35K35}

\vspace{4mm}

\section{Introduction}

When computing the solution of a partial differential equation in
an unbounded domain, one often introduces artificial boundaries.
In order to limit the computational cost, these boundaries must be
chosen not too far from the domain of interest. Therefore, the
boundary conditions must be good approximations to the so-called
transparent boundary condition (i.e., such that the solution of
the problem in the bounded domain is equal to the solution in the
original domain).

One of the numerical methods for the solution of problems in
unbounded domains is the Dirichlet-to-Neumann (DtN) Finite
Elements Method. Its name comes from the fact that it involves the
nonlocal Dirichlet-to-Neumann (DtN) map on an artificial boundary
which encloses the computational domain. Originally the method has
been developed for the solution of linear elliptic problems, such
as wave scattering problems governed by Helmholtz equation or by
the equations of time-harmonic elasticity. Recently, the method
has been extended in a number of directions, and further analyzed
and improved, by the D. Givoli \cite{12} and others. In
\cite{12}-\cite{16} it can be find more references to various
numerical methods associated with solving similar problems.

This question is of crucial interest in such different areas as
geophysics, plasma physics, fluid dynamics (see \cite{1} -
\cite{5}).

The need for practical transparent (artificial) boundary
conditions combining efficiency and simplicity is evident. Such
conditions must satisfy several criteria: (i) The resulting
initial boundary value problem should be unique and stable; (ii)
the solution to the initial boundary value problem should coincide
or closely approximate the solution of the infinite problem on the
definitional domain of the boundary value problem; and (iii) the
conditions must allow for an analytical solution or an efficient
numerical implementation.

In this paper we consider artificial boundary conditions for the
high-order Cauchy problem for the heat equation. The conditions
satisfy the above-mentioned criteria (i), (ii) and (iii). Similar
results were taken for the Laplace equation in \cite{6} and for
high-order Laplace equation (polyharmonic equations) in \cite{7,
17}. And a transfer of the Sommerfeld radiation condition to a
boundary of bounded domains for the Helmholtz operator is studied
in work \cite{18}.

The transparent boundary condition is usually an integral relation
in time and space between $u$ and its normal derivative on the
boundary, which makes it impractical from a numerical point of
view.
 Alternatively, the requirement for boundary conditions can
be avoided when the solution of a partial differential equation is
approximated in the form of convolution in space and time with the
fundamental solution. An efficient approximation of this type for
the heat equation is proposed in \cite{8}.

\vspace{5mm}

\section{Main results}

In what follows, we shall need in definitions of the Holder spaces
(see, e.g., \cite{9}, pp. 33-34 and 117-118).

On a cylindrical domain  $\Omega\equiv Q\times(0,T)$, where
$Q\subset \mathbb R^{n}, n\in \mathbb N$ is a simply-connected
bounded domain with a sufficiently smooth boundary $\partial Q$,
we consider the following heat potential
\begin{equation}
u(x,t)=\int^{t}_{0}\int_{Q}\varepsilon_{m,n}(x-\xi,
t-\tau)f(\xi,\tau)d\xi d\tau, \label{eq-11}
\end{equation}
where $\varepsilon_{m,n}(x,t)=\frac{\theta(t)t^{m-1}}{(2\sqrt{\pi
t})^{n}}\exp(-\frac{\mid x\mid^{2}}{4t}), m\in \mathbb N$ is a
fundamental solution of the Cauchy problem for the high-order heat
equation, i.e. this solves the following high-order heat equation
$$
\diamondsuit_{x,t}^{m}\varepsilon_{m,n}(x-\xi,
t-\tau)=(\frac{\partial}{\partial
t}-\Delta_{x})^{m}\varepsilon_{m,n}(x-\xi,t-\tau)=0
$$
and its adjoint
$$
(\diamondsuit_{\xi,\tau}^{+})^{m}\varepsilon_{m,n}(x-\xi,
t-\tau)=(-\frac{\partial}{\partial
\tau}-\Delta_{\xi})^{m}\varepsilon_{m,n}(x-\xi,t-\tau)=0
$$
for all $t>\tau$ and $\xi,x\in R^{n}.$

Indeed, it's not difficult to check that for $m>1$
$$
\diamondsuit_{x,t}\varepsilon_{m,n}(x-\xi,
t-\tau)=\varepsilon_{m-1,n}(x-\xi, t-\tau)
$$
and
$$
\diamondsuit_{\xi,\tau}^{+}\varepsilon_{m,n}(x-\xi,
t-\tau)=\varepsilon_{m-1,n}(x-\xi,t-\tau).
$$

From [17, chapter 1] for $x\in Q$ the condition
$$
\lim_{\tau\rightarrow
t}\int_{Q}(\diamondsuit_{\xi,\tau}^{+})^{m-1}\varepsilon_{m,n}(x-\xi,
t-\tau)f(\xi,\tau)d\xi=f(x,t)
$$
is valid.

It is easy to see that if $f(x,t)\in
C^{\gamma,\frac{\gamma}{2}}_{x,t}(\overline{\Omega})$ then
$u(x,t)\in
C^{2m+\gamma,m+\frac{\gamma}{2}}_{x,t}(\overline{\Omega})$, where
$0<\gamma<1$ and
\begin{equation}
\diamondsuit^{m}u(x,t)=f(x,t), (x,t)\in\Omega, \label{eq-12}
\end{equation}
with initial conditions
\begin{equation}
\frac{\partial^{i}u(x,t)}{\partial t^{i}}\mid_{t=0}=0,
i=\overline{0,m-1}. \label{eq-13}
\end{equation}
We note that if the function $u(x,t)\in
C^{2m+\gamma,m+\frac{\gamma}{2}}_{x,t}(\overline{\Omega})$
satisfies the conditions (\ref{eq-13}) then it will satisfy to the
following conditions too
$$\diamondsuit^{i}u(x,t)\mid_{t=0}=0, i=\overline{0, m-1}.$$

Consider the problem (\ref{eq-12})-(\ref{eq-13}) with non-local
type boundary conditions

\begin{multline}
I^{k}_{u}(x,t)\mid_{x\in\partial Q}\equiv-\frac{\diamondsuit^{k}
u(x,t)}{2}\\
+\sum_{i=0}^{m-k-1}[\int^{t}_{0}\int_{\partial Q}\frac{\partial
(\diamondsuit^{+}_{\xi,\tau})^{i+k}\varepsilon_{m,n}(x-\xi,
t-\tau)}{\partial n_{\xi}}\diamondsuit^{m-i-1}
u(\xi,\tau)dS_{\xi}d\tau \\
-\int^{t}_{0}\int_{\partial Q}\frac{\partial \diamondsuit^{m-i-1}
u(\xi,\tau)}{\partial
n_{\xi}}(\diamondsuit^{+}_{\xi,\tau})^{i+k}\varepsilon_{m,n}(x-\xi,
t-\tau)dS_{\xi}d\tau]=0, \\
k=\overline{0,m-1}, (x,t)\in \partial Q\times (0,T), \label{eq-14}
\end{multline}

where $\frac{\partial}{\partial n_{\xi}}$ denotes the exterior
normal derivative on the boundary $\partial Q$.

\vspace{5mm}

\begin{theorem} \label{th-2} The heat potential for the high-order heat
equation (\ref{eq-12}) is an unique classical solution of  the
non-local type initial boundary value problem
(\ref{eq-12})-(\ref{eq-14}).
\end{theorem}

\begin{proof} The following equalities are valid
$$
(\diamondsuit_{\xi,\tau}^{+})^{m}\varepsilon_{m,n}(x-\xi,
t-\tau)=(\diamondsuit_{\xi,\tau}^{+})^{m-1}\varepsilon_{m-1,n}(x-\xi,
t-\tau)=...=\diamondsuit_{\xi,\tau}^{+}\varepsilon_{1,n}(x-\xi,
t-\tau)=0,
$$
for all $t>\tau$ and $\xi,x\in R^{n}$,
$$
\lim_{\tau\rightarrow
t}(\diamondsuit_{\xi,\tau}^{+})^{k}\varepsilon_{m,n}(x-\xi,
t-\tau)=0,
$$
for $k<m-1$. Therefore
$$
\lim_{\alpha\rightarrow
0+0}\int_{Q}(\diamondsuit^{+}_{\xi,\tau})^{k}\varepsilon_{m,n}(x-\xi,
\alpha)u(\xi,t-\alpha)d\xi=0
$$
for $k<m-1$ and
$$
\int^{t}_{0}\int_{Q}(\diamondsuit^{+}_{\xi,\tau})^{m}\varepsilon_{m,n}(x-\xi,
t-\tau)u(\xi,\tau)d\xi d\tau=\lim_{\alpha\rightarrow
0+0}\int^{t-\alpha}_{0}\int_{Q}(\diamondsuit^{+}_{\xi,\tau})^{m}\varepsilon_{m,n}(x-\xi,
t-\tau)u(\xi,\tau)d\xi d\tau
$$
$$
=\lim_{\alpha\rightarrow
0+0}\int^{t-\alpha}_{0}\int_{Q}0u(\xi,\tau)d\xi d\tau=0.
$$

We assume that $u(x,t)\in
C^{2m+\gamma,m+\frac{\gamma}{2}}_{x,t}(\overline{\Omega})$. Let us
denote $\diamondsuit^{0}\equiv I,$ where $I$ is identity operator.
A direct calculation shows that
$$
u(x,t)=\lim_{\alpha\rightarrow
0+0}\int^{t-\alpha}_{0}\int_{Q}\varepsilon_{m,n}(x-\xi,
t-\tau)f(\xi,\tau)d\xi d\tau
$$
$$
=\lim_{\alpha\rightarrow
0+0}\int^{t-\alpha}_{0}\int_{Q}\varepsilon_{m,n}(x-\xi,
t-\tau)\diamondsuit (\diamondsuit^{m-1} u(\xi,\tau))d\xi d\tau
$$
$$
=\lim_{\alpha\rightarrow 0+0}[\int_{Q}\varepsilon_{m,n}(x-\xi,
t-\tau)\diamondsuit^{m-1} u(\xi,\tau)\mid^{t-\alpha}_{0}d\xi$$
$$
-\int^{t-\alpha}_{0}\int_{Q}\frac{\partial
\varepsilon_{m,n}(x-\xi, t-\tau)}{\partial\tau}\diamondsuit^{m-1}
u(\xi,\tau)d\xi d\tau
$$
$$
-\int^{t-\alpha}_{0}\int_{Q}\varepsilon_{m,n}(x-\xi,
t-\tau)\Delta_{\xi}(\diamondsuit^{m-1} u(\xi,\tau))d\xi d\tau]
$$
$$=\lim_{\alpha\rightarrow
0+0}[\int^{t-\alpha}_{0}\int_{Q}\diamondsuit^{+}_{\xi,\tau}\varepsilon_{m,n}(x-\xi,
t-\tau)\diamondsuit^{m-1}u(\xi,\tau)d\xi d\tau$$
$$+\int^{t-\alpha}_{0}\int_{\partial Q}\frac{\partial
\varepsilon_{m,n}(x-\xi, t-\tau)}{\partial
n_{\xi}}\diamondsuit^{m-1} u(\xi,\tau)dS_{\xi} d\tau$$
$$-\int^{t-\alpha}_{0}\int_{\partial Q}\frac{\partial
\diamondsuit^{m-1} u(\xi,\tau)}{\partial
n_{\xi}}\varepsilon_{m,n}(x-\xi, t-\tau)dS_{\xi} d\tau]=...$$
$$=\lim_{\alpha\rightarrow
0+0}[\int_{Q}(\diamondsuit^{+})^{m-1}\varepsilon_{m,n}(x-\xi,
t-\tau) u(\xi,\tau)\mid^{t-\alpha}_{0}d\xi$$
$$+\int^{t-\alpha}_{0}\int_{Q}(\diamondsuit^{+}_{\xi,\tau})^{m}\varepsilon_{m,n}(x-\xi,
t-\tau)u(\xi,\tau)d\xi d\tau]$$
$$+\sum_{i=0}^{m-1}[\int^{t}_{0}\int_{\partial Q}\frac{\partial
(\diamondsuit^{+}_{\xi,\tau})^{i}\varepsilon_{m,n}(x-\xi,
t-\tau)}{\partial n_{\xi}}\diamondsuit^{m-i-1} u(\xi,\tau)dS_{\xi}
d\tau$$
$$-\int^{t}_{0}\int_{\partial Q}\frac{\partial \diamondsuit^{m-i-1}
u(\xi,\tau)}{\partial
n_{\xi}}(\diamondsuit^{+}_{\xi,\tau})^{i}\varepsilon_{m,n}(x-\xi,
t-\tau)dS_{\xi}d\tau]$$
$$=u(x,t)+\sum_{i=0}^{m-1}[\int^{t}_{0}\int_{\partial
Q}\frac{\partial
(\diamondsuit^{+}_{\xi,\tau})^{i}\varepsilon_{m,n}(x-\xi,
t-\tau)}{\partial n_{\xi}}\diamondsuit^{m-i-1} u(\xi,\tau)dS_{\xi}
d\tau$$
$$-\int^{t}_{0}\int_{\partial Q}\frac{\partial \diamondsuit^{m-i-1}
u(\xi,\tau)}{\partial
n_{\xi}}(\diamondsuit^{+}_{\xi,\tau})^{i}\varepsilon_{m,n}(x-\xi,
t-\tau)dS_{\xi}d\tau]$$ for any $(x,t)\in\Omega$.

Here we get
\begin{multline}
\sum_{i=0}^{m-1}[\int^{t}_{0}\int_{\partial Q}\frac{\partial
(\diamondsuit^{+}_{\xi,\tau})^{i}\varepsilon_{m,n}(x-\xi,
t-\tau)}{\partial n_{\xi}}\diamondsuit^{m-i-1} u(\xi,\tau)dS_{\xi}
d\tau\\
-\int^{t}_{0}\int_{\partial Q}\frac{\partial \diamondsuit^{m-i-1}
u(\xi,\tau)}{\partial
n_{\xi}}(\diamondsuit^{+}_{\xi,\tau})^{i}\varepsilon_{m,n}(x-\xi,
t-\tau)dS_{\xi}d\tau]=0,\forall(x,t)\in\Omega. \label{eq-15}
\end{multline}

Applying properties of the double layer and single layer
potentials  (see \cite{10} or \cite{19}) to (\ref{eq-15}) as
$(x,t)\rightarrow
\partial Q\times(0,T)$, we obtain
\begin{align*}
-\frac{u(x,t)}{2}+\sum_{i=0}^{m-1}[\int^{t}_{0}\int_{\partial
Q}\frac{\partial
(\diamondsuit^{+}_{\xi,\tau})^{i}\varepsilon_{m,n}(x-\xi,
t-\tau)}{\partial n_{\xi}}\diamondsuit^{m-i-1} u(\xi,\tau)dS_{\xi}
d\tau\\
-\int^{t}_{0}\int_{\partial Q}\frac{\partial \diamondsuit^{m-i-1}
u(\xi,\tau)}{\partial
n_{\xi}}(\diamondsuit^{+}_{\xi,\tau})^{i}\varepsilon_{m,n}(x-\xi,
t-\tau)dS_{\xi}d\tau]=0,\forall(x,t)\in\partial Q\times(0,T).
\end{align*}

Applying the differential expression $\diamondsuit^{k}_{x, t},\,\,
k=\overline{1,m-1}$ to $u(x,t),$ we get
$$
\diamondsuit^{k}_{x, t}u(x,t)=\diamondsuit^{k}_{x,
t}\left(\lim_{\alpha\rightarrow
0+0}\int^{t-\alpha}_{0}\int_{Q}\varepsilon_{m,n}(x-\xi,
t-\tau)f(\xi,\tau)d\xi d\tau\right)
$$

$$
=\diamondsuit^{k-1}_{x, t}\Bigl(\lim_{\alpha\rightarrow
0+0}\int_{Q}\varepsilon_{m,n}(x-\xi, \alpha)f(\xi,t-\alpha)d\xi
$$
$$+\lim_{\alpha\rightarrow
0+0}\int^{t-\alpha}_{0}\int_{Q}\varepsilon_{m-1,n}(x-\xi,
t-\tau)f(\xi,\tau)d\xi d\tau\Bigr)
$$
$$
=\diamondsuit^{k-1}_{x, t}\left(\lim_{\alpha\rightarrow
0+0}\int^{t-\alpha}_{0}\int_{Q}\varepsilon_{m-1,n}(x-\xi,
t-\tau)f(\xi,\tau)d\xi d\tau\right)=...
$$
$$
=\lim_{\alpha\rightarrow
0+0}\int^{t-\alpha}_{0}\int_{Q}\varepsilon_{m-k,n}(x-\xi,
t-\tau)f(\xi,\tau)d\xi d\tau
$$
and after similar calculations as above, we get the following
boundary conditions for each $k$

\begin{align*}
I^{k}_{u}(x,t)\equiv-\frac{\diamondsuit^{k}u(x,t)}{2}+\sum_{i=0}^{m-k-1}[\int^{t}_{0}\int_{\partial
Q}\frac{\partial
(\diamondsuit^{+}_{\xi,\tau})^{i+k}\varepsilon_{m,n}(x-\xi,
t-\tau)}{\partial n_{\xi}}\diamondsuit^{m-i-1} u(\xi,\tau)dS_{\xi}
d\tau\\
-\int^{t}_{0}\int_{\partial Q}\frac{\partial \diamondsuit^{m-i-1}
u(\xi,\tau)}{\partial
n_{\xi}}(\diamondsuit^{+}_{\xi,\tau})^{i+k}\varepsilon_{m,n}(x-\xi,
t-\tau)dS_{\xi}d\tau]=0\\
\forall(x,t)\in\partial Q\times(0,T),k=\overline{0,m-1}.
\end{align*}

Thus, the heat potential (\ref{eq-11}) satisfies the boundary
conditions (\ref{eq-14}).

Conversely, if a function $u_{1}(x,t)\in
C^{2m+\gamma,m+\frac{\gamma}{2}}_{x,t}(\overline{\Omega})$
satisfies equation (\ref{eq-12}), the initial condition
(\ref{eq-13}) and boundary conditions (\ref{eq-14}) then
$u_{1}(x,t)=u(x,t)$, where $u(x,t)$ is the heat potential
(\ref{eq-11}). If this is not so, then a function
$v(x,t)=u_{1}(x,t)-u(x,t)$ satisfies the homogeneous equation
\begin{equation}
\diamondsuit^{m}v(x,t)=0, (x,t)\in\Omega, \label{eq-16}
\end{equation}
with the initial conditions
\begin{equation}
\diamondsuit^{k}v(x,t)\mid_{t=0}=0, x\in\Omega,
k=\overline{0,m-1}, \label{eq-17}
\end{equation}
and  the boundary conditions
\begin{equation}
I^{k}_{v}(x,t)=0,\forall(x,t)\in\partial
Q\times(0,T),k=\overline{0,m-1}. \label{eq-18}
\end{equation}
Where we use the following notations
\begin{multline}
I^{k}_{v}(x,t)\equiv\sum_{i=0}^{m-k-1}[\int^{t}_{0}\int_{\partial
Q}\frac{\partial
(\diamondsuit^{+}_{\xi,\tau})^{i+k}\varepsilon_{m,n}(x-\xi,
t-\tau)}{\partial n_{\xi}}\diamondsuit^{m-i-1} v(\xi,\tau)dS_{\xi}
d\tau\\
-\int^{t}_{0}\int_{\partial Q}\frac{\partial \diamondsuit^{m-i-1}
v(\xi,\tau)}{\partial
n_{\xi}}(\diamondsuit^{+}_{\xi,\tau})^{i+k}\varepsilon_{m,n}(x-\xi,
t-\tau)dS_{\xi}d\tau]=0,
\\
\forall(x,t)\in\Omega,k=\overline{0,m-1}. \label{eq-19}
\end{multline}
On the other hand, by using (\ref{eq-16}) and (\ref{eq-17}), we
obtain following equalities
$$
0=\int^{t}_{0}\int_{Q}\varepsilon_{m,n}(x-\xi, t-\tau)\cdot0d\xi
d\tau$$
$$=\int^{t}_{0}\int_{Q}\varepsilon_{m,n}(x-\xi,
t-\tau)\diamondsuit^{m-k} (\diamondsuit^{k} v(\xi,\tau))d\xi d\tau
$$
$$
=\diamondsuit^{k} v(x,t)+I_{v}^{k}(x,t),
\forall(x,t)\in\Omega,k=\overline{0,m-1}.
$$

Applying properties of the double layer potential (see \cite{10},
\cite{11} or \cite{19}) to (\ref{eq-19}) as $(x,t)\rightarrow
\partial Q\times(0,T)$, we have
$$
\diamondsuit^{k}v(x,t)\mid_{\partial
Q\times(0,T)}=-I^{k}_{v}(x,t)\mid_{\partial
Q\times(0,T)}=0,k=\overline{0,m-1}.
$$
I.e. the problem (\ref{eq-16})-(\ref{eq-18}) is equivalent to
\begin{equation}
\diamondsuit^{m}v(x,t)=0, (x,t)\in\Omega, \label{eq-20}
\end{equation}
with the initial conditions
\begin{equation}
\diamondsuit^{k}v(x,t)\mid_{t=0}=0, x\in Q, k=\overline{0,m-1},
\label{eq-21}
\end{equation}
and the boundary conditions
\begin{equation}
\diamondsuit^{k}v(x,t)\mid_{\partial
Q\times(0,T)}=0,k=\overline{0,m-1}. \label{eq-22}
\end{equation}
From (\ref{eq-20})-(\ref{eq-22}) with $k=m-1$ using the Maximum
Principle we get
$$\diamondsuit^{m-1}v(x,t)=0, (x,t)\in\Omega.$$
And repeating the similar calculations we have
$$\diamondsuit^{i} v(x,t)=0,\forall(x,t)\in
\bar{\Omega},i=\overline{0,m-1},$$ i.e.,
$v(x,t)=u_{1}(x,t)-u(x,t)$ and $u_{1}(x,t)=u(x,t)$.

This completes the proof of Theorem \ref{th-2}.
\end{proof}

\vspace{5mm}

\begin{remark}\label{rem-1}
The heat kernel, i.e. the fundamental solution of the high-order
heat equation $\varepsilon_{m,n}(x,t)$ is the Green function for
the non-local type boundary value problem
(\ref{eq-12})-(\ref{eq-14}).
\end{remark}

Let us consider the heat equation
\begin{equation}
\diamondsuit u(x,t)=f(x,t), (x,t)\in\Omega, \label{eq-b15}
\end{equation}
with initial condition
\begin{equation}
u(x,0)=0 \label{eq-b16}
\end{equation}
and with inhomogeneous non-local type boundary condition
\begin{multline}
-\frac{u(x,t)}{2}+\int^{t}_{0}\int_{\partial Q}\frac{\partial
\varepsilon_{1,n}(x-\xi, t-\tau)}{\partial n_{\xi}}
u(\xi,\tau)dS_{\xi}d\tau \\
-\int^{t}_{0}\int_{\partial Q}\frac{\partial u(\xi,\tau)}{\partial
n_{\xi}}\varepsilon_{1,n}(x-\xi, t-\tau)dS_{\xi}d\tau=\varphi(x,
t),
\\ (x,t)\in
\partial Q\times (0,T), \label{eq-b17}
\end{multline}
where $\varphi(x,t)\in
C^{2+\gamma,1+\frac{\gamma}{2}}_{x,t}(\overline{\partial
Q}\times[0, T]),$ $\varphi(x,0)=0, x\in\partial Q$ and
$\frac{\partial}{\partial n_{\xi}}$ denotes the exterior normal
derivative on the boundary $\partial Q$.

\vspace{5mm}

\begin{remark}\label{rem-2}
In case $\varphi(x, t)\equiv 0,$ from the theorem \ref{th-2}
follows that an unique solution of the problem
(\ref{eq-b15})-(\ref{eq-b17}) is given by potential
\begin{equation}
u(x,t)=\int^{t}_{0}\int_{Q}\varepsilon_{1,n}(x-\xi,
t-\tau)f(\xi,\tau)d\xi d\tau. \label{eq-b11}
\end{equation}
\end{remark}

\vspace{5mm}

\begin{lemma}\label{l-1}
 The function
\begin{equation}u(x,t)=\int^{t}_{0}\int_{\partial Q}\frac{\partial G(x, \xi, t,
\tau)}{\partial n_{\xi}}\varphi(\xi,\tau)d\xi d\tau \label{eq-c1}
\end{equation}
is an unique solution of the following problem
\begin{equation}
\diamondsuit u(x,t)=0, (x,t)\in\Omega, \label{eq-c2}
\end{equation}
with initial condition
\begin{equation}
u(x,0)=0 \label{eq-c3}
\end{equation}
and with inhomogeneous non-local type boundary condition
\begin{multline}
-\frac{u(x,t)}{2}+\int^{t}_{0}\int_{\partial Q}\frac{\partial
\varepsilon_{1,n}(x-\xi, t-\tau)}{\partial n_{\xi}}
u(\xi,\tau)dS_{\xi}d\tau \\
-\int^{t}_{0}\int_{\partial Q}\frac{\partial u(\xi,\tau)}{\partial
n_{\xi}}\varepsilon_{,n}(x-\xi, t-\tau)dS_{\xi}d\tau=\varphi(x,
t),
\\ (x,t)\in
\partial Q\times (0,T), \label{eq-c4}
\end{multline}
where $\frac{\partial}{\partial n_{\xi}}$ denotes the exterior
normal derivative on the boundary $\partial Q$ and $G(x, \xi, t,
\tau)$ is the Green function of the problem
(\ref{eq-b15})-(\ref{eq-b16}) with the Dirichlet boundary
condition.
\end{lemma}

\begin{proof}
$$
0=\int^{t}_{0}\int_{Q}\varepsilon_{1,n}(x-\xi, t-\tau)0d\xi
d\tau=\int^{t}_{0}\int_{Q}\varepsilon_{1,n}(x-\xi,
t-\tau)\diamondsuit u(\xi,\tau)d\xi d\tau=
$$
$$
=u(x,t)+\int^{t}_{0}\int_{Q}\diamondsuit^{+}_{\xi,\tau}\varepsilon_{1,n}(x-\xi,
t-\tau) u(\xi,\tau)d\xi d\tau-\int^{t}_{0}\int_{\partial
Q}\frac{\partial \varepsilon_{1,n}(x-\xi, t-\tau)}{\partial
n_{\xi}} u(\xi,\tau)dS_{\xi} d\tau
$$
$$+\int^{t}_{0}\int_{\partial Q}\varepsilon_{1,n}(x-\xi, t-\tau)\frac{\partial u(\xi,\tau)}{\partial
n_{\xi}}dS_{\xi}d\tau=u(x,t)-\int^{t}_{0}\int_{\partial
Q}\frac{\partial \varepsilon_{1,n}(x-\xi, t-\tau)}{\partial
n_{\xi}} u(\xi,\tau)dS_{\xi} d\tau
$$
$$
+\int^{t}_{0}\int_{\partial Q}\varepsilon_{1,n}(x-\xi,
t-\tau)\frac{\partial u(\xi,\tau)}{\partial n_{\xi}}dS_{\xi}d\tau
$$
for all $(x,t)\in Q\times(0,T).$ Hence,

$$
u(x,t)=\int^{t}_{0}\int_{\partial Q}\frac{\partial
\varepsilon_{1,n}(x-\xi, t-\tau)}{\partial n_{\xi}}
u(\xi,\tau)dS_{\xi} d\tau-\int^{t}_{0}\int_{\partial
Q}\varepsilon_{1,n}(x-\xi, t-\tau)\frac{\partial
u(\xi,\tau)}{\partial n_{\xi}}dS_{\xi}d\tau
$$
for all $(x,t)\in Q\times(0,T).$ And
$$
u(x,t)=-\frac{u(x,t)}{2}+\int^{t}_{0}\int_{\partial
Q}\frac{\partial \varepsilon_{1,n}(x-\xi, t-\tau)}{\partial
n_{\xi}} u(\xi,\tau)dS_{\xi} d\tau
$$
$$
-\int^{t}_{0}\int_{\partial Q}\varepsilon_{1,n}(x-\xi,
t-\tau)\frac{\partial u(\xi,\tau)}{\partial n_{\xi}}dS_{\xi}d\tau
$$
as $(x,t)\rightarrow\partial Q\times(0,T).$ Since,
$$
u(x,t)=-\frac{u(x,t)}{2}+\int^{t}_{0}\int_{\partial
Q}\frac{\partial \varepsilon_{1,n}(x-\xi, t-\tau)}{\partial
n_{\xi}} u(\xi,\tau)dS_{\xi} d\tau
$$

$$
-\int^{t}_{0}\int_{\partial Q}\varepsilon_{1,n}(x-\xi,
t-\tau)\frac{\partial u(\xi,\tau)}{\partial
n_{\xi}}dS_{\xi}d\tau=\varphi(x,t)
$$
for all $(x,t)\in \partial Q\times(0,T).$ So, problem
(\ref{eq-c2})-(\ref{eq-c4}) equivalent to the problem for
homogeneous heat equation (\ref{eq-c2}) with initial condition
(\ref{eq-c3}) and with the following Dirichlet boundary condition
$$
u(x,t)=\varphi(x,t)
$$
for all $(x,t)\in \partial Q\times(0,T).$ The lemma is proved.
\end{proof}

From lemma \ref{l-1} and remark \ref{rem-2} follows statement of
the following theorem.

\vspace{5mm}

\begin{theorem} \label{th-b2}
The solution of the problem (\ref{eq-b15})-(\ref{eq-b17}) is given
by the following formula
\begin{equation}
u(x,t)=\int^{t}_{0}\int_{Q}\varepsilon_{1,n}(x-\xi,
t-\tau)f(\xi,\tau)d\xi d\tau+\int^{t}_{0}\int_{\partial
Q}\frac{\partial G(x, \xi, t, \tau)}{\partial
n_{\xi}}\varphi(\xi,\tau)d\xi d\tau, \label{eq-b21}
\end{equation}
where $G(x, \xi, t, \tau)$ is the Green function of the problem
(\ref{eq-b15})-(\ref{eq-b16}) with the Dirichlet boundary
condition.
\end{theorem}

\vspace{15mm}

Durvudkhan Suragan,

Institute of mathematics and mathematical modeling, st. Shevchenko
28, Almaty, Kazakhstan and Al-Farabi Kazakh National University,
ave. al-Farabi 71, Almaty, Kazakhstan.

E-mail: suragan@list.ru

\vspace{5mm}

Niyaz Tokmagambetov,

Institute of mathematics and mathematical modeling, st. Shevchenko
28, Almaty, Kazakhstan and Al-Farabi Kazakh National University,
ave. al-Farabi 71, Almaty, Kazakhstan.

E-mail: tokmagam@list.ru

\end{document}